\documentclass[a4paper,11pt,reqno,oneside]{amsart}

\usepackage{hyperref}
\usepackage[headinclude,DIV13]{typearea}
\areaset{15.1cm}{25.0cm}
\usepackage{amsfonts,amssymb,amsmath,amsthm,bbm}
\usepackage{cancel} 
\usepackage[latin1] {inputenc}
\usepackage{graphicx, psfrag}
\usepackage{txfonts}
\usepackage{graphicx}
\usepackage{color}

\newtheorem{theorem}{Theorem}

\newtheorem{proposition}[theorem]{Proposition}

\theoremstyle{definition}
\newtheorem{definition}[theorem]{Definition}

\theoremstyle{remark}

\newcommand{\ind}[1]{\mathbbm{1}_{\left[ {#1} \right] }}

\newcommand{\calA}{\mathcal{A}}

\newcommand{\bbR}{\mathbb{R}}

\newcommand{\bbZ}{\mathbb{Z}}

\usepackage{cite}
\def\BibTeX{{\rm B\kern-.05em{\sc i\kern-.025em b}\kern-.08em
		T\kern-.1667em\lower.7ex\hbox{E}\kern-.125emX}}

\title[Note on the absence of a wetting transition for a pinned GFF in $d\geq3$]{Note on Bolthausen-Deuschel-Zeitouni's paper on the \\
	Absence of a wetting transition for a pinned harmonic crystal \\
	in dimensions
three and larger}

\author[L. Coquille]{Loren Coquille}
\address{L. Coquille\\Institut Fourier,
	UMR 5582 du CNRS\\
	Universit\'e de Grenoble Alpes \\
	100 rue des Math\'ematiques \\
	38610 Gi\`eres, France}
\email{loren.coquille@univ-grenoble-alpes.fr}

\author[P. Mi\l{}o\'s]{Piotr Mi\l{}o\'s}
\address{P. Mi\l{}o\'s\\
	MIMUW, Banacha 2, 02-097 Warszawa, Poland
}
\email{pmilos@mimuw.edu.pl}

\begin{document}
\maketitle

\begin{abstract}
	The article \cite{BDZ} provides a proof of the absence of a wetting transition for the discrete Gaussian free field conditioned to stay positive, and undergoing a weak delta-pinning at height 0. The proof is generalized to the case of a square pinning-potential replacing the delta-pinning, but it relies on a lower bound on the probability for the field to stay above the support of the potential, the proof of which appears to be incorrect. We provide a modified proof of the absence of a wetting transition in the square-potential case, which does not require the aforementioned lower bound. An alternative approach is given in a recent paper by Giacomin and Lacoin \cite{GiaLac}.
\end{abstract}

\section{Definitions and notations}

We keep the notations of \cite{BDZ} except for the field which we call $\phi$ instead of $X$. Let $A$ be a finite subset of $\bbZ^d$, let $\phi=(\phi_{x})_{x\in \bbZ^d}\in\bbR^{\bbZ^d}$ and the Hamiltonian defined as
\begin{equation}
H_A(\phi) = \frac1{8d}\sum_{x,y\in A\cup\partial A\, :\, |x-y|=1}(\phi_{x}-\phi_{y})^2
\end{equation}
where $\partial A$ is the outer boundary of $A$. The following probability measure on $\bbR^A$ defines the discrete Gaussian free field on $A$ (with zero boundary condition) :
\begin{equation}
P_A(d\phi)=\frac1{Z_A}e^{-H_A(\phi)}d\phi_A\delta_0(d\phi_{A^c})
\end{equation}
where $d\phi_A=\prod_{x\in A}d\phi_{x}$ and $\delta_0$ is the Dirac mass at 0. The partition function $Z_A$ is the normalization $Z_A=\int_{\bbR^A}\exp(-H(\phi_A))d\phi_A$.
We will also need the following definition of a set $A$ being $\Delta$-sparse (morally meaning that it has only one pinned point per cell of side-length $\Delta$), which we reproduce from \cite[page 1215]{BDZ} : 
\begin{definition}\label{def-delta-sparse}
Let $N\in\bbZ$, $\Delta>0$, $\Lambda_N=\{-\lfloor N\rfloor/2,\ldots,\lfloor N\rfloor/2\}^d$ and let $l_N^\Delta=\{z_i\}_{i=1}^{|l_N^\Delta|}$ denote a finite collection of points $z_i\in\Lambda_N$ such that for each $y\in\Lambda_N\cap\Delta\bbZ^d$ there is exacly one $z\in l_N^\Delta$ such that $|z-y|<\Delta/10$. Let $A_{l_N^\Delta}=\Lambda_N\backslash l_N^\Delta$.
\end{definition}

\section{Lower bound on the probability of the hard wall condition}
The proof of \cite[Theorem 6]{BDZ} relies on \cite[Proposition 3]{BDZ}. Unfortunately, the proof provided in the paper, when applied with $t>0$ provides a lower bound which is a little bit weaker than what is claimed, namely

\begin{proposition}\label{new-prop3}{Correction of \cite[Proposition 3]{BDZ}} :\\
Assume $d\geq3$ and let $t\geq0$. Then there exist three constants $c_1,c_2,c_3>0$ depending on $t$, and $c_4>0$ independent of $t$, such that, for all $\Delta$ integer large enough
\begin{eqnarray}
\liminf_{N\to\infty}\inf_{l^\Delta_N}\frac1{(2N+1)^d}\log P_{A_{l^\Delta_N}}(X_i\geq t,i\in A_{l^\Delta_N})\geq
-\frac{d\log\Delta}{\Delta^d}+c_1\frac{\log\log\Delta}{\Delta^d}
-\frac{c_2e^{c_4 t\sqrt{\log\Delta}}}{\Delta^{d}(\log\Delta)^{c_3}}
\end{eqnarray}
\end{proposition}

The statement of \cite[Proposition 3]{BDZ} only contains the first two terms. The dependence in $t$ vanishes between equations $(2.4)$ and $(2.5)$ in \cite{BDZ}. Note that for $t=0$  the third term is irrelevant and the bound coincides with the one stated in the paper.

\section{Proof of the absence of a wetting transition in the square-potential case}

Let us introduce the following notations 
\begin{align*}
\hat\xi_N&=\sum_{x\in\Lambda_N}\ind{|\phi_x|\leq a}, \quad \tilde\xi_N=\sum_{x\in\Lambda_N}\ind{\phi_x\in [0,a]},\\
\Omega_A^+&=\{\phi_x\geq0,\,\forall x\in A\}, \quad \Omega_N^+=\{\phi_x\geq0,\,\forall x\in \Lambda_N\}\\
\calA &=\{x\in\Lambda_N : \phi_x\in[0,a]\}
\end{align*}
and the following probability measure with square-potential pinning :
$$ \tilde
P_{N,a,b}(d\phi)=\frac1{\tilde Z_{N,a,b}}\exp\left(-H(\phi)+\sum_{x\in\Lambda_N}b\ind{\phi_x\in[0,a]}\right)d\phi_{\Lambda_N}\delta_0(d\phi_{\Lambda_N^c})$$
in contrast with the measure used in \cite{BDZ} :
$$ \hat
P_{N,a,b}(d\phi)=\frac1{\hat Z_{N,a,b}}\exp\left(-H(\phi)+\sum_{x\in\Lambda_N}b\ind{\phi_x\in[-a,a]}\right)d\phi_{\Lambda_N}\delta_0(d\phi_{\Lambda_N^c}).$$
Observe that
$$ \tilde
P_{N,a,b}(\tilde\xi_N<\epsilon N^d | \Omega_N^+)=\hat
P_{N,a,b}(\hat\xi_N<\epsilon N^d | \Omega_N^+)$$

\begin{theorem}{(Absence of wetting transition, \cite[Theorem 6]{BDZ})}\\
	Assume $d\geq3$ and let $a,b>0$ be arbitrary. Then there exists $\epsilon_{b,a},\eta_{b,a}>0$ such that 
	\begin{align}
	 \tilde
	P_{N,a,b}(\tilde\xi_N>\epsilon_{b,a} N^d | \Omega_N^+)\geq 1-\exp(-\eta_{b,a}N^d).
	\end{align}
	provided $N$ is large enough.
\end{theorem}

\begin{proof}
	Let us compute the probability of the complementary event and provide bounds on the numerator and the denominator corresponding to the conditional probability :
\begin{equation}\label{proba}
\tilde
P_{N,a,b}(\tilde \xi_N<\epsilon N^d | \Omega_N^+)=\frac{\tilde
P_{N,a,b}(\{\tilde \xi_N<\epsilon N^d \}\cap \Omega_N^+)}{\tilde
P_{N,a,b}( \Omega_N^+)}
\end{equation}

\subsection{Lower bound on the denominator}

Writing 
\begin{equation}
\exp(\sum_{x\in\Lambda_N}b\ind{\phi_x\in[0,a]})=\prod_{x\in\Lambda_N}((e^b-1)\ind{\phi_{x}\in[0,a]}+1)
\end{equation}
and using the FKG inequality, we get

\begin{align}
\tilde
P_{N,a,b}( \Omega_N^+)&\stackrel{FKG}{\geq}\frac{Z_N}{\tilde
  Z_{N,a,b}}\sum_{A\subset\Lambda_N}(e^b-1)^{|A|}
\underbrace{P_N(\calA\supset A)}_{(*)}
\underbrace{P_N(\Omega_{A^c}^+|\calA\supset A)}_{(**)}
\underbrace{P_N(\Omega_{A}^+|\calA\supset A)}_{=1}.
\end{align}
Let us first bound the term $(**)$:
\begin{align}
(**)&=P_N(\phi\geq0 \,on\,A^c|\phi\in[0,a] \,on\, A)
=\int_{[0,a]^{A}}
P_N(\phi\geq0 \,on\,A^c|
\phi=\psi \, on \, A)g(\psi)d\psi
\end{align}
for some density function $g$. Let $\tilde{\psi}$ be the harmonic extension of $\psi$ to $\Lambda_N\backslash A$. Since $\tilde{\psi}\geq0$, we have
\begin{align}
(**)&=\int_{[0,a]^{A}}
P_N(\phi+\tilde{\psi}\geq0 \,on\,A^c|
\phi=0 \, on \, A)g(\psi)d\psi\\
&=\int_{[0,a]^{A}}
P_{A^c}(\phi+\tilde{\psi}\geq0 \,on\,A^c)g(\psi)d\psi\\
&\geq P_{A^c}(\Omega_{A^c}^+)
\end{align}
For the term $(*)$, we write  $A=\{x_1,\ldots,x_{|A|}\},$ and $A_i=\{x_{i+1},\ldots,x_{|A|}\},$
\begin{align}
(*)&=P_N(\phi\in[0,a]\,on\, A)\\
&=\prod_{i=1}^{|A|}P_N(\phi_{x_i}\in[0,a] |
\phi_{x_{i+1}},\ldots,\phi_{x_{|A|}}\in[0,a])\\
&=\prod_{i=1}^{|A|}\int_{[0,a]^{A_i}}
P_N(\phi_{x_i}\in[0,a] |
\phi=\psi \, on \, A_i)g_i(\psi)d\psi
\end{align}
for some density function $g_i$. Let $\tilde{\psi}$ be the harmonic extension of $\psi$ to $\Lambda_N\backslash A_i$, we have
\begin{align}
(*)&=\prod_{i=1}^{|A|}\int_{[0,a]^{A_i}}
P_N(\phi_{x_i}+\tilde{\psi}_{x_i}\in[0,a] |
\phi=0 \, on \, A_i)g_i(\psi)d\psi\\
&=
\prod_{i=1}^{|A|}\int_{[0,a]^{A_i}}
P_{A_i^c}(\phi_{x_i}+\tilde{\psi}_{x_i}\in[0,a])g_i(\psi)d\psi
\\
&\label{ineq}\geq \prod_{i=1}^{|A|}P_{A_i^c}(\phi_{x_i}\in[0,a] )\\
&\geq [c(1/2\wedge a)]^{|A|}
\end{align}
for some $c=c(d)>0$, since the variance of the free field is bounded
in $d\geq3$. The inequality \eqref{ineq} comes from the fact that $P_{A_i^c}(\phi_{x_i}+\tilde{\psi}_{x_i}\in[0,a])\geq P_{A_i^c}(\phi_{x_i}\in[0,a])$ since $\tilde{\psi}_{x_i}\in[0,a]$ and $\phi_{x_i}$ is a centered Gaussian variable.\\
Hence,
\begin{align}
\tilde
P_{N,a,b}( \Omega_N^+)&{\geq}\frac{Z_N}{\tilde
  Z_{N,a,b}}\sum_{A\subset\Lambda_N}\exp(J'{|A|})P_{A^c}(\Omega_{A^c}^+)
\end{align}
with $J'=\log(e^b-1)+\log c+\log(1/2\wedge a)$.

\subsection{Upper bound on the numerator}

\begin{align}
\tilde
P_{N,a,b}(\{\tilde \xi_N<\epsilon N^d \}\cap \Omega_N^+)&=
\frac{Z_N}{\tilde
  Z_{N,a,b}}\sum_{A : |A|<\epsilon N^d}(e^b-1)^{|A|}
\underbrace{P_N(\calA\supset A)}_{\leq(1/2\wedge a)^{|A|}}
\underbrace{P_N(\Omega_{N}^+|\calA\supset
  A)}_{\leq1}\\
&\leq \frac{Z_N}{\tilde
  Z_{N,a,b}} \sharp\{A : |A|<\epsilon N^d\}\exp(J\epsilon N^d)
\end{align}
with $J=\log(e^b-1)+\log(1/2\wedge a)$, where $\sharp X$ denotes the cardinality of the set $X$.

\subsection{Upper bound on \eqref{proba}}

\begin{align}
\tilde
P_{N,a,b}(\tilde \xi_N<\epsilon N^d | \Omega_N^+)\leq
\frac{\exp(J\epsilon N^d)\sharp\{A : |A|<\epsilon
  N^d\}}{\sum_{A\subset\Lambda_N}\exp(J'{|A|})P_{A^c}(\Omega_{A^c}^+)}
\end{align}
And now we proceed similarily as for the proof with $\delta$-pinning potential:
\begin{align}
\frac1{N^d}\log \tilde
P_{N,a,b}(\tilde \xi_N<\epsilon N^d | \Omega_N^+)\leq
&\frac1{N^d}\log \left(\exp(J\epsilon N^d)\sharp\{A : |A|<\epsilon
  N^d\}\right)\label{stirling}\\
&-\frac1{N^d}\log \sum_{A\subset\Lambda_N}\exp(J'{|A|})P_{A^c}(\Omega_{A^c}^+)\label{prop3}
\end{align}
The right hand side of \eqref{stirling} can be bounded by
$\epsilon(J+1-\log\epsilon)$ as $N$ tends to infinity (by a rough
approximation and the Stirling formula), which in turn
can be made as close to 0 as we want by choosing
$\epsilon=\epsilon(J)$ sufficiently small. See \cite{BDZ}.\\

To bound \eqref{prop3} we use \cite[Proposition 3]{BDZ} \underline{with
  $t=0$} which matches to our Proposition \ref{new-prop3} :
\begin{align}
\eqref{prop3}&\leq -\frac1{N^d}\log \sum_{A\subset\Lambda_N \,:\, A \,is\, \Delta-sparse }\exp(J'{|A|})P_{A^c}(\Omega_{A^c}^+)\\
&\leq -\frac1{N^{d}} \left( \left(\frac
    N\Delta\right)^d[(d\log\Delta+c_0)+J'
  -d\log\Delta+c_1\log\log\Delta]\right)\\
&=-\frac{J'+c_0+c_1\log\log\Delta}{\Delta^d}<0 \mbox{ for
  $\Delta=\Delta(J')$ large enough.}
\end{align}
where $\Delta$-sparseness corresponds to Definition \ref{def-delta-sparse} : a set $A\subset\Lambda_N$ is $\Delta$-sparse if it equals $A_{l_N^\Delta}$, for some set $l_N^\Delta$.

\end{proof}

\end{document}